\documentclass[10pt,a4paper]{scrartcl}

\usepackage[left=3cm,right=4cm,top=3cm,bottom=5cm,includeheadfoot]{geometry}
\usepackage[english]{babel}
\usepackage[utf8]{inputenc}
\usepackage{inputenc, amsmath, amssymb, latexsym, graphicx, amsthm}
\usepackage{mathtools, mathabx,xcolor}
\usepackage[normalem]{ulem}
\usepackage[mathscr]{euscript}
\usepackage[headsepline]{scrlayer-scrpage}

\newtheorem{theorem}{Theorem}[section]
\newtheorem{corollary}[theorem]{Corollary}
\newtheorem{proposition}[theorem]{Proposition}
\newtheorem{lemma}[theorem]{Lemma}

\newtheorem{claim}[theorem]{Claim}
\theoremstyle{definition}
\newtheorem{definition}[theorem]{Definition}

\newtheorem*{example*}{Example}
\newtheorem{remark}[theorem]{Remark}

\newcommand{\T}{\ensuremath{\mathbb{T}}}

\newcommand{\R}{\ensuremath{\mathbb{R}}}
\newcommand{\Z}{\ensuremath{\mathbb{Z}}}
\newcommand{\N}{\ensuremath{\mathbb{N}}}

\newcommand{\eps}{\varepsilon}
\newcommand{\abs}[1]{\left|#1\right|}

\newcommand{\ad}{\mathrm{ad}}
\newcommand{\dsep}{\Delta}
\newcommand{\ac}{\ensuremath{\mathrm{{ac}}}}
\newcommand{\oac}{\ensuremath{\overline{\mathrm{ac}}}}
\newcommand{\uac}{\ensuremath{\underline{\mathrm{ac}}}}
\newcommand{\Sep}{\ensuremath{\mathrm{Sep}}}

\newcommand{\diam}{\ensuremath{\mathrm{diam}}}
\newcommand{\Fol}{\ensuremath{\mathcal F}}
\newcommand{\odim}{\overline\dim}

\newcommand{\mc}{\mathcal}
\renewcommand{\phi}{\varphi}
\newcommand{\ssq}{\subseteq}
\renewcommand{\=}{\coloneqq}

\newcommand{\lam}{\lambda}

\newcommand{\roundqed}{{\hfill \Large $\circ$}}

\newcommand{\equi}{\ensuremath{\Leftrightarrow}}

\newcommand{\ld}{\ensuremath{,\ldots,}}

\newcommand{\dist}{\ensuremath{\mathop{\textrm{dist}}}}
\newcommand{\ind}{\ensuremath{\mathbf{1}}}

\newcommand{\interior}{\ensuremath{\mathrm{int}}}
\newcommand{\inte}{\ensuremath{\mathrm{int}}}

\newcommand{\alphlist}{\begin{list}{(\alph{enumi})}{\usecounter{enumi}\setlength{\parsep}{2pt}
      \setlength{\itemsep}{1pt} \setlength{\topsep}{5pt}
      \setlength{\partopsep}{3pt}}}
\newcommand{\arablist}{\begin{list}{(\arabic{enumi})}{\usecounter{enumi}\setlength{\parsep}{2pt}
          \setlength{\itemsep}{1pt} \setlength{\topsep}{5pt}
          \setlength{\partopsep}{3pt}}}
\newcommand{\romanlist}{\begin{list}{(\roman{enumi})}{\usecounter{enumi}\setlength{\parsep}{2pt}
              \setlength{\itemsep}{1pt} \setlength{\topsep}{5pt}
              \setlength{\partopsep}{3pt}}}
\newcommand{\Romanlist}{\begin{list}{(\Roman{enumi})}{\usecounter{enumi}\setlength{\parsep}{2pt}
              \setlength{\itemsep}{1pt} \setlength{\topsep}{5pt}
              \setlength{\partopsep}{3pt}}}
\newcommand{\bulletlist}{\begin{list}{$\bullet$}{\setlength{\parsep}{2pt}
                \setlength{\itemsep}{1pt} \setlength{\topsep}{5pt}
                \setlength{\partopsep}{3pt}\setlength{\leftmargin}{15pt}}} 
\newcommand{\Alphlist}{\begin{list}{(\Alph{enumi})}{\usecounter{enumi}\setlength{\parsep}{2pt}
      \setlength{\itemsep}{1pt} \setlength{\topsep}{5pt}
      \setlength{\partopsep}{3pt}}}
 \newcommand{\listend}{\end{list}}

\newcommand{\cD}{\mathcal{D}}

\newcommand{\cL}{\mathcal{L}}

\newcommand{\cP}{\mathcal{P}}

\newcommand{\jLim}{\ensuremath{\lim_{j\rightarrow\infty}}}

\newcommand{\oplam}{\mbox{\Large $\curlywedge$}}
\newcommand{\vol}{\ensuremath{\mathrm{Vol}}}
\newcommand{\Span}{\ensuremath{\mathrm{Span}}}
\newcommand{\Dim}{\ensuremath{\mathrm{dim}}}
\newcommand{\uboxdim}{\ensuremath{\overline{\Dim}_\mathrm{B}}}

\makeatletter
\let\@fnsymbol\@arabic
\makeatother

\title{\textbf\Large Amorphic complexity of group actions with applications to quasicrystals}

\author{
    \normalsize Gabriel Fuhrmann\thanks{Department of Mathematical Sciences, 
		Durham University, UK. Email: {\texttt gabriel.fuhrmann@durham.ac.uk}}
    \and\normalsize Maik Gröger\thanks{Faculty of Mathematics and Computer Science, Jagiellonian University in Krak\'{o}w,
        Poland. Email: {\texttt maik.groeger@im.uj.edu.pl}}
    \and\normalsize Tobias Jäger\thanks{Department of Mathematics, University of Jena,
        Germany. Email: {\texttt tobias.jaeger@uni-jena.de}}
    \and\normalsize Dominik Kwietniak\thanks{Faculty of Mathematics and Computer Science, Jagiellonian University in Krak\'{o}w,
        Poland. Email: {\texttt dominik.kwietniak@uj.edu.pl}}}
\date{}

\begin{document}

\maketitle
\begin{abstract}
    In this article, we define amorphic complexity for actions of 
	locally compact $\sigma$-compact amenable groups on compact metric spaces.
	Amorphic complexity, originally introduced for $\Z$-actions, is a 
	topological invariant which measures the complexity of dynamical systems
	in the regime of zero entropy.
	We show that it is tailor-made to study strictly ergodic group actions with
	discrete spectrum and continuous eigenfunctions.
	This class of actions includes, in particular, Delone dynamical systems
	related to regular model sets obtained via Meyer's cut and project method.
	We provide sharp upper bounds on amorphic complexity of such systems.
	In doing so, we observe an intimate relationship between amorphic
	complexity and fractal geometry.
\end{abstract}

\section{Introduction}\label{sec: Introduction}

The study of low-complexity notions for group actions is both a timely and a classical topic.
Its roots go back to Halmos, Mackey, and von Neumann
who classified actions with discrete spectrum, as well as Auslander,
Ellis, Furstenberg, and Veech who set the foundations of the theory of equicontinuous
actions and their extensions.
Recent years have seen plenty of progress in illuminating the richness of
possible dynamical behaviour of minimal actions of general groups in the low complexity regime, see for example \cite{Krieger2007,CortezPetite2008,CortezMedynets2016,SabokTsankov2017,
Glasner2018,LcackaStraszak2018,FuhrmannKwietniak2020}.
As a matter of fact, the investigation of this regime not only contributes to the understanding of group actions as such but is 
of fundamental importance in the understanding of aperiodic order---with further applications to geometry, number theory and harmonic analysis \cite{Meyer1972,BaakeGrimm2013}---and the diffraction spectra of so-called Delone sets, that is, mathematical models of physical quasicrystals.
The latter results from the observation that diffraction spectra of Delone sets can be studied by means of certain associated Delone dynamical systems \cite{LeeMoody2006,BaakeLenzMoody2007,Lenz2009}, see also \cite{BaakeGrimm2013} for further information and references.
Analysing these Delone dynamical systems, it is most natural to ask when two such systems are conjugate \cite{KellendonkSadun2014}.
The standard operating procedure to answer this question clearly is to utilize dynamical invariants and one might be tempted
to study topological entropy of Delone dynamics.
However, the physically most interesting case of pure point diffraction
turns out to necessarily come with zero entropy \cite{BaakeLenzRichard2007}.
There is hence a need for finer topological
invariants which can distinguish zero entropy systems.

In this article, we propose amorphic complexity---a notion recently introduced for $\Z$-actions \cite{FuhrmannGroegerJaeger2016}---as a promising candidate for this purpose.
To that end, we extend amorphic complexity to actions of locally compact, 
$\sigma$-compact and amenable groups.
We will see that amorphic complexity is tailor-made to study strictly ergodic systems with discrete spectrum and continuous eigenfunctions, that is, minimal mean equicontinuous systems \cite[Corollary~1.6]{FuhrmannGroegerLenz2022}.
Most importantly, however, we show that amorphic complexity is not only theoretically well-behaved but also well-computable in specific examples.
This is particularly true due to a neat connection to fractal geometry.
We elaborate on this in the last section of this article where we apply
our findings to model sets---particular Delone sets constructed by means of 
Meyer's cut and project method \cite{Meyer1972}.
Note that the relation between fractal geometry and ergodic theory is a
well-established field of research which goes back to Billinglsey and 
Furstenberg \cite{Billingsley1960,Furstenberg1967}.
For results on this relation in the context of quasicrystals and, more
generally, aperiodic order, we refer the interested reader, for instance, to 
\cite{Fogg2002,KellendonkLenzSavinien2015,Julien2017} and references therein. 

Before we introduce amorphic complexity and discuss our main results in more detail, let us briefly clarify some basic terminology.
A triple $(X,G,\alpha)$ is called a \emph{(topological) dynamical system}
if $X$ is a compact metric space (endowed with a metric $d$), $G$ is a 
topological group and $\alpha$ is a continuous action of $G$ on $X$ by
homeomorphisms (continuity of $\alpha$ is understood as continuity of the
map $G\times X \ni(g,x)\mapsto \alpha(g)(x)\in X$).
In the following, we use the shorthand $gx$ instead of 
$\alpha(g)(x)$ for the action of $g\in G$ on $x\in X$ via $\alpha$.
Likewise, we may occasionally keep the action $\alpha$ implicit and simply refer to $(X,G)$ as a dynamical system.

As mentioned before, we throughout assume that $G$
is locally compact, $\sigma$-compact and amenable.
Recall that there is hence a \emph{(left) Følner sequence}, that is, a sequence
$(F_n)_{n\in\N}$ of compact subsets of $G$ having positive Haar measure such that
\begin{align}\label{eq:Foelner_seq}
    \lim_{n\to\infty}\frac{m(KF_n\triangle F_n)}{m(F_n)}
        =0\quad\textnormal{for every compact }K\subseteq G,
\end{align}
where $\triangle$ denotes the symmetric difference and $m$ is a
\emph{(left) Haar measure} of $G$ (we may synonymously write
$\abs{F}$ for the Haar measure $m(F)$ of a measurable set $F\ssq G$) \cite[Theorem~3.2.1]{EmersonGreenleaf1967}.
We will also make use of the existence of \emph{right Følner sequences} which fulfil a condition analogous to \eqref{eq:Foelner_seq} with the left Haar
measure and the multiplication from the left replaced by the
right Haar measure and multiplication from the right, respectively.
However, we would like to stress that in the following, each F\o lner sequence
is assumed to be a left F\o lner sequence if not stated otherwise.
Given a (left or right) F\o lner sequence $\mc F=(F_n)$, the \emph{(upper) asymptotic density} of a measurable subset $E\subseteq G$
with respect to $\Fol$ is defined as
\begin{equation}\label{eq: upper density}
	\ad_{\Fol}(E)=\varlimsup\limits_{n\to\infty}\frac{\abs{E\cap F_n}}{\abs{F_n}}.
\end{equation}

Let us next turn to the definition of amorphic complexity
of a dynamical system $(X,G)$ with respect to a F\o lner sequence $\Fol=(F_n)_{n\in\N}$ in $G$.
Given $x,y\in X$, $\delta>0$, we set
\begin{equation*}
	\dsep(X,G,\delta,x,y)=\left\{t\in G\;|\;d(tx,ty)\geq\delta\right\}.
\end{equation*}
For $\nu\in(0,1]$, we say that $x$ and $y$ are \emph{$(\delta,\nu)$-separated}
with respect to $\Fol$ if 
\begin{equation*}
	\ad_{\Fol}(\dsep(X,G,\delta,x,y))=\varlimsup_{n\to\infty}\frac{\abs{\dsep(X,G,\delta,x,y)\cap F_n}}{\abs{F_n}}\geq\nu.
\end{equation*}
Accordingly, a subset $S\subseteq X$ is said to be \emph{$(\delta,\nu)$-separated} with
respect to $\Fol$ if all distinct points $x,y\in S$ are $(\delta,\nu)$-separated.  
This already yields the first key notion in this work: the {\em (asymptotic) separation number of $(X,G)$ with respect to $\delta>0$ and $\nu\in(0,1]$}, denoted by $\Sep_{\Fol}(X,G,\delta,\nu)$, is the supremum over the cardinalities of all $(\delta,\nu)$-separated sets in $X$.

In general, the asymptotic separation numbers do not have to be finite (even
though $X$ is compact) which immediately gives the following dichotomy: if
$\Sep_{\Fol}(X,G,\delta,\nu)$ is finite for all $\delta,\nu>0$, we say $(X,G)$
has {\em finite separation numbers} with respect to $\Fol$ otherwise, we say it
has {\em infinite separation numbers}.
The first result which we would like to highlight identifies canonical classes
of group actions with infinite and finite separation numbers, respectively.
The corresponding proofs can all be found in Section~\ref{sec: finite vs infinite sep numbers}.

\begin{theorem}
	Suppose $(X,G)$ is a dynamical system with $X$ a compact metric space
	and $G$ a locally compact, $\sigma$-compact and amenable group.
	\mbox{}
	\begin{enumerate}
		\item[(i)]
			If $(X,G)$ is weakly mixing with respect to a
			non-trivial $G$-invariant probability measure, then $(X,G)$ has infinite
			separation numbers with respect to every F\o lner sequence.  Likewise, if
			$G$ allows for a uniform lattice and $(X,G)$ has positive topological
			entropy, then $(X,G)$ has infinite separation numbers with respect to every
			F\o lner sequence.
		\item[(ii)]
			If $G$ is unimodular, that is, there is a two-sided (left and right)
			F\o lner sequence and $(X,G)$ is minimal, then $(X,G)$ has finite
			separation numbers with respect to every F\o lner
			sequence if and only if $(X,G)$ is mean equicontinuous.
	\end{enumerate}
\end{theorem}

It is worth mentioning that the class of mean equicontinuous systems comprises
all Delone dynamical systems associated to regular model sets, see also
Section~\ref{sec: ac and regular model sets}.  For further examples of mean
equicontinuous actions of groups different from $\Z$, we refer the reader to the
literature
\cite{Robinson1996,Robinson1999,Cortez2006,Vorobets2012,Garcia-Ramos2017,Glasner2018,LcackaStraszak2018,FuhrmannKwietniak2020,GroegerLukina2021,FuhrmannGroegerLenz2022}.

If $(X,G)$ has finite separation numbers, we are in a position to obtain finer
information by studying the scaling behaviour of the separation numbers as the
separation frequency $\nu$ tends to zero.  Here, we may in principle consider
arbitrary growth rates.  So far, however, previous results indicate that
polynomial growth is the most relevant, see
\cite{FuhrmannGroegerJaeger2016,GroegerJaeger2016,FuhrmannGroeger2020} for
$G=\Z$.  With this in mind, we define the \emph{lower} and \emph{upper amorphic
  complexity} of $(X,G)$ with respect to $\Fol$ as
\begin{equation*}
	\uac_{\Fol}(X,G)=\adjustlimits\sup_{\delta>0}\varliminf_{\nu\to 0}
		\frac{\log \Sep_{\Fol}(X,G,\delta,\nu)}{-\log \nu}
	\quad\textnormal{and}\quad
	\oac_{\Fol}(X,G)=\adjustlimits\sup_{\delta>0}\varlimsup_{\nu\to 0}
		\frac{\log\Sep_{\Fol}(X,G,\delta,\nu)}{-\log \nu}. 
\end{equation*}
In case that both values coincide, we call $\ac_{\Fol}(X,G)=\uac_{\Fol}(X,G)=
\oac_{\Fol}(X,G)$ the {\em amorphic complexity} of $(X,G)$ with respect to $\Fol$.
It is convenient to set $\ac_{\Fol}(X,G)=\infty$ if $(X,G)$ has infinite separation numbers 
with respect to $\Fol$.
We discuss the most basic properties of amorphic complexity---including its invariance under conjugacy---in Section~\ref{sec: basic props ac}.

The next question we address is to which extent the asymptotic
separation numbers and amorphic complexity depend on the particular Følner
sequence $\Fol$.  
In general, we cannot rule out different amorphic complexities
with respect to different F\o lner sequences.  
In fact, this problem already
occurs when $G=\Z$, see Section~\ref{sec: independence Folner seq}.  With the
next theorem, however, we provide a sufficient criterion for the independence
from $\Fol$.  Here, we say a dynamical system $(X,G)$ is \emph{pointwise
  uniquely ergodic} if every orbit closure is uniquely ergodic.  A strengthening
of the following statement and its proof can be found in Section~\ref{sec:
  independence Folner seq}.
\begin{theorem}
 Let $(X,G)$ be a dynamical system whose product $(X^2,G)$ is pointwise uniquely ergodic.
 Then $(X,G)$ has infinite separation numbers with respect to some F\o lner sequence if and only if it has infinite separation numbers with respect to all F\o lner sequences.
 Moreover, $\oac_\Fol(X,G)$ and $\uac_\Fol(X,G)$ are independent of the particular F\o{}lner sequence $\Fol$.
\end{theorem}
It is worth mentioning that mean equicontinuous systems verify the assumptions of the above theorem
\cite[Theorem~1.2]{FuhrmannGroegerLenz2022}.

Finally, we apply amorphic complexity to the dynamics of
regular model sets.  Before we come to the precise formulation, we need to
introduce some terminology.  In doing so, we restrict to a rather brief
description of the most essential notions and refer the reader to
Section~\ref{sec: ac and regular model sets} for the details.  A \emph{cut and
  project scheme} is a triple $(G,H,\cL)$, where $G$ and $H$ are locally compact
abelian groups and $\cL$ is an irrational lattice in $G\times H$.  Together with
a compact subset $W=\overline{\operatorname{int}(W)}\ssq H$---referred to as a
\emph{window}---$(G,H,\mc L)$ defines
a particular instance of a Delone set, a so-called \emph{model} set
  \[
  \oplam(W) = \pi_G((G\times W)\cap \mc L),
  \]
  where $\pi_G:G\times H\to G$ denotes the canonical projection. We call $W$ as
well as $\oplam(W)$ \emph{regular} if $\partial W$ is of zero Haar measure and
say $W$ is \emph{irredundant} if $ \{ h \in H\;|\;h+W=W \} = \{0\}$.  Now, as
$\oplam(W)$ is a subset of $G$, $G$ naturally acts on $\oplam(W)$ by
translations.  It turns out that the closure of all translated copies of
$\oplam(W)$ is compact (in a suitable topology on subsets of $G$).  Denoting
this closure by $\Omega(\oplam(W))$, we arrive at the Delone dynamical system
$(\Omega(\oplam(W)),G)$ associated to the model set $\oplam(W)$.  We obtain

\begin{theorem}
  Let $(G,H,\cL)$ be a cut and project scheme with $W\ssq H$ a regular irredundant window and suppose $G$ and $H$ are second countable.
  Then for every  Følner sequence $\Fol$ in $G$, we get
  \begin{equation*}
    \oac_{\Fol}(\Omega(\oplam(W)),G)\leq\frac{\uboxdim(H)}{\uboxdim(H)-\uboxdim(\partial W)},
  \end{equation*}
  assuming that $\uboxdim(H)$ is finite.
\end{theorem}
Here, $\uboxdim(\cdot)$ denotes the upper box dimension, see Section \ref{sec: ac and regular model sets} for the details.
Let us remark that we further show that the above estimates are sharp in that they are realised by particular model sets.
In conclusion, we obtain that every value in
$[1,\infty)$ can be attained by amorphic complexity of minimal systems.

Motivated by the above results, we finish with the following question.
\begin{center}
	\textit{Given a locally compact, $\sigma$-compact and amenable
	group acting minimally on a compact metric space.
	Which values can amorphic complexity attain?}
\end{center}
In particular, for minimal $\Z$- or $\R$-actions, we conjecture that amorphic
complexity cannot take values in $(0,1)$.
Indeed, this complexity gap was recently established for subshifts associated
to primitive constant length substitutions \cite{FuhrmannGroeger2020}
and is a classical phenomenon which is well known to occur for polynomial
entropy of minimal symbolic subshifts. 
For non-minimal $\Z$-actions, however, it was recently shown that all values
in $(0,1)$ can be obtained by amorphic complexity, see \cite{Kulczycki2021,Kulczycki2022}.

\subsubsection*{Acknowledgments}

This project has received funding from the European Union's Horizon 2020
research and innovation program under the Marie Sk\l{}odowska-Curie grant agreement
No 750865.
The research leading to these results has received funding from the Norwegian
Financial Mechanism 2014-2021 via the POLS grant no. 2020/37/K/ST1/02770.
Furthermore, it received support by the DFG Emmy-Noether grant Ja 1721/2-1
and DFG Heisenberg grant Oe 538/6-1.
DK was supported by the National Science Centre, Poland, grant no. 2018/29/B/ST1/01340. 
GF, MG and DK would like to thank the Mathematisches Forschungsinstitut Oberwolfach
for its enormous hospitality during a Research in Pairs stay (R1721) at
the MFO in October 2017 where many ideas of this work were developed.
This work was finished during a visit
of GF and MG to the Jagiellonian University in Kraków in September
2020, which was also supported by the National Science Centre, Poland,
grant no.\ 2018/29/B/ST1/01340.


\section{Basic properties of amorphic complexity}\label{sec: basic props ac}

In this section, we collect the most basic properties of amorphic complexity.
In particular, given a group $G$ which allows for a lattice $\mc L$,
we discuss how amorphic complexity of a $G$-action relates to amorphic complexity of the associated $\mc L$-action.

The proof of the following statement is verbatim as the proofs of
 \cite[Proposition~3.4 \& Proposition~3.9]{FuhrmannGroegerJaeger2016}.
\begin{proposition}\label{prop: ac topological invariant and power rule}
    Let $(X,G)$ and $(Y,G)$ be dynamical systems.
    We have:
    \begin{enumerate}
	\item[(a)] If $(Y,G)$ is a factor of $(X,G)$, then
	\[
		\uac_{\Fol}(Y,G)\leq\uac_{\Fol}(X,G)
		\quad\textnormal{and}\quad
		\oac_{\Fol}(Y,G)\leq\oac_{\Fol}(X,G).
	\]
	In particular, (upper and lower) amorphic complexity
	is a topological invariant.
	\item[(b)] We have that
	\begin{align*}
		\uac_{\Fol}(X\times Y,G)\geq \uac_{\Fol}(X,G)+\uac_{\Fol}(Y,G),\quad
		\oac_{\Fol}(X\times Y,G)\leq\oac_{\Fol}(X,G)+\oac_{\Fol}(Y,G).
	\end{align*}
	In particular, if $\ac_\Fol(X,G)$ and $\ac_\Fol(Y,G)$ exist, then $\ac_\Fol(X\times Y,G)$ exists as well.
    \end{enumerate}
\end{proposition}

Before we proceed with further properties of amorphic complexity,
we take a closer look at certain particularly well-behaved F\o{}lner sequences.
Recall that a \emph{van Hove sequence} $(A_n)_{n\in \N}$ in $G$
is a sequence of compacta $A_n\subseteq G$ of positive Haar measure such that
\[
 \lim_{n\to \infty}\frac{m\big(\partial_K A_n)}{m(A_n)}=0,
\]
for every compact set $K\subseteq G$ with $e\in K$,
where $\partial_K A_n\=KA_n\setminus \operatorname{int}\big(\bigcap_{g\in K} gA_n\big)$
(see \cite[Appendix~3]{Tempelman1992} and \cite{Strungaru2005} for
further reference).
It is not hard to see that every van Hove sequence is a Følner sequence.
In fact, it holds
\begin{proposition}[{\cite[Appendix 3.K]{Tempelman1992}}]\label{prop: tempelman}
 Let $G$ be a locally compact $\sigma$-compact amenable topological group.
 A sequence $(A_n)$ of compact subsets of $G$ is a van Hove sequence if and only if
 it is a F\o{}lner sequence and
 \begin{align}\label{eq: sufficient condition for Folner to be Hove}
 \lim_{n\to \infty}\frac{m(\partial_U A_n)}{m(A_n)}=0,
\end{align}
for some open neighbourhood $U$ of the neutral element $e$ in $G$.
\end{proposition}
\begin{remark}\label{rem: discrete groups and van Hove}
 In particular, if $G$ is discrete, then every F\o{}lner sequence in $G$ is, in fact,
 a van Hove sequence.
\end{remark}

It is well known that every locally compact $\sigma$-compact amenable group
allows for a van Hove sequence.
For the convenience of the reader, we prove the following (possibly well-known) refinement of this statement which we need in the sequel.

\begin{proposition}\label{prop: Folner=van Hove}
 Let $G$ be a locally compact $\sigma$-compact amenable topological group.
 Suppose $(F_n)$ is a F\o{}lner sequence in $G$ and $B$ is a compact
 neighbourhood of $e$.
 Then $A_n\=BF_n$ defines a van Hove sequence in $G$ with
 $\ad_{(A_n)}(E) = \ad_{(F_n)}(E)$ for every measurable $E\ssq G$.
\end{proposition}
\begin{proof}
 The last part follows from $E\cap A_n\ssq (E\cap F_n)\cup (F_n\triangle A_n)$ and
  \begin{align}\label{eq: An = Fn}
   0\leq\lim_{n\to\infty}m(A_n\triangle F_n)/m(A_n)\leq\lim_{n\to\infty}m(BF_n\triangle F_n)/m(F_n)=0,
  \end{align}
 which is a consequence of the fact that $(F_n)$ is a F\o{}lner sequence and
 $F_n\ssq BF_n=A_n$.
 
 For the first part, we make use of Proposition~\ref{prop: tempelman}.
 To that end,  observe that for every (compact) $K\ssq G$ we have 
 $KA_n \triangle A_n \ssq (KA_n \triangle F_n) \cup (F_n \triangle A_n)
 =(KBF_n \triangle F_n) \cup (F_n \triangle A_n)$.
 Due to \eqref{eq: An = Fn}
 and the fact that $(F_n)$ is a F\o lner sequence, this gives
 that $(A_n)$ is a F\o lner sequence, too.
 To see \eqref{eq: sufficient condition for Folner to be Hove}, we need the following
 \begin{claim}
  There is a relatively compact open neighbourhood $U$ of $e$ such that $F_n\ssq \operatorname{int}\big(\bigcap_{g\in U} gA_n\big)$ for each $n\in \N$.
 \end{claim}
\proof[Proof of the claim]
First, observe that $\operatorname{int}\big(\bigcap_{g\in U} gBF_n\big)\supseteq\operatorname{int}\big( \bigcap_{g\in U} gB\big )F_n$.
To prove the claim, it hence suffices to show
that there is $U$ with $e\in \operatorname{int}\big( \bigcap_{g\in U} gB\big )$.

For a contradiction, suppose $e\in\overline{\bigcup_{g\in U} g B^c}$ for every $U$ in the open neighbourhood filter $\mc U$ of $e$.
In other words, suppose there is a net $(g_U)_{U\in \mc U}$ with $g_U\in U$ (so that $g_U\to e$)  and a net $(h_U)_{U\in \mc U}$ in $B^c$ such that $g_Uh_U\to e$.
This, however, implies $h_U\to e$ which contradicts $e\in\operatorname{int}(B)$.
Therefore, there is $U\in \mc U$ with $e\in\operatorname{int}\big( \bigcap_{g\in U} gB\big )$.
Clearly, $U$ can be chosen open and relatively compact.
\roundqed\smallskip
 
Now, pick some $U$ as in the above claim.
As $(F_n)$ is a F\o{}lner sequence, we have 
\begin{align*}
	m(\partial_U A_n)/m(A_n)\leq m(UA_n\setminus F_n)/m(F_n)\leq m(\overline{U}BF_n\setminus F_n)/m(F_n) \stackrel{n\to\infty}{\longrightarrow} 0.
\end{align*}
Finally, it follows from Proposition~\ref{prop: tempelman} that $(A_n)$ is a van Hove
sequence.
\end{proof}

For the next statement, recall that a \emph{uniform lattice} $\cL$ in $G$ is
a discrete subgroup of $G$ such that there exists a measurable precompact
subset $C\ssq G$, referred to as \emph{fundamental domain}, 
with $G=\bigsqcup_{\lam\in \cL}C\lam$ and $m(C)>0$.
With the lattice $\mc L$ being a subgroup of $G$, we have a naturally
defined dynamical system $(X,\mc L)$ and it turns out that amorphic
complexity is well behaved when going from $(X,G)$ over to $(X,\mc L)$.
\begin{lemma}\label{lem: ac of group and lattice coincide}
	Assume $(X,G)$ is a dynamical system and $G$ allows for a uniform
	lattice $\cL$.
	Then for every F\o{}lner sequence $\Fol$ in $G$ there is a F\o{}lner
	sequence $\Fol'$ in $\cL$ such that
	\begin{align*}
		\uac_{\Fol}(X,G)=\uac_{\Fol'}(X,\cL)\qquad \textnormal{and}\qquad
		\oac_{\Fol}(X,G)=\oac_{\Fol'}(X,\cL).
	\end{align*}
	Furthermore, $(X,G)$ has infinite separation numbers with respect to $\mc F$
	if and only if $(X,\mc L)$ has infinite separation numbers with respect to $\mc F'$.
\end{lemma}
\begin{proof}
	We denote the Haar measure on $G$ by $m$ and that on $\cL$ by
        $|\cdot|$. Let $C\ssq G$ be a fundamental domain as in the above
        definition of a uniform lattice.  First, observe that for all $\delta>0$
        there are $\delta^-_\delta,\delta^+_\delta>0$ such that
        for all $x,y\in X$ and $c\in C$ we have
          $d(c^{-1}x,c^{-1}y)\geq \delta^-_\delta$ whenever $d(x,y)\geq \delta$
          and $d(cx,cy)\geq \delta^+_\delta$ whenever $d(x,y)\geq
          \delta^-_\delta$.
        This straightforwardly follows from the precompactness of $C$.
	
	Further, due to Proposition~\ref{prop: Folner=van Hove}, we may assume without loss of generality that
	$\Fol$ is a van Hove sequence.
	Under this assumption, there are van Hove sequences $\Fol'=(F_n')$ and $\Fol''=(F_n'')$ 
	in $\cL$ with $\lim_{n\to \infty} |F_n'|/|F_n''|=1$
	such that $CF_n'$ and $CF_n''$ are von Hove sequences in $G$ and
	$CF_n'\ssq F_n\ssq CF_n''$, see for example \cite[Lemma~3.2]{Hauser2021}.
	We will show that for all $x,y\in X$ and $\delta>0$ we have
	\begin{align}\label{eq: comparison asymptotic densities in G and lattice}
		\ad_\Fol(\dsep(X,G,\delta,x,y))\leq \ad_{\Fol'}(\dsep(X,\cL,\delta^-_\delta,x,y)) 
		\leq \ad_\Fol(\dsep(X,G,\delta^+_{\delta},x,y)).
	\end{align}
	Clearly, this implies that for all $\nu\in (0,1)$ and all $\delta>0$
	\[
		\Sep_\Fol(X,G,\delta,\nu)\leq \Sep_{\Fol'}(X,\cL,\delta^-_\delta,\nu)\leq \Sep_\Fol(X,G,\delta^+_{\delta},\nu)
	\]
	and thus proves the statement.
	
	By definition of $\delta^-_\delta$ and $\delta^+_{\delta}$ and
        since $C$ is a fundamental domain, we have
	\begin{align*}
	\dsep(X,G,\delta,x,y)\subseteq C \dsep(X,\cL,\delta^-_\delta,x,y)\subseteq \dsep(X,G,\delta^+_{\delta},x,y).
	\end{align*}
	Hence, utilizing the fact that for any subset $F\ssq\cL$ we have $m(CF)=|F|\cdot m(C)$, 
	we obtain \eqref{eq: comparison asymptotic densities in G and lattice} from the following computation
	\begin{align*}
		\ad_\Fol(\dsep(X,G,\delta,x,y))&=\varlimsup_{n\to\infty}m(\dsep(X,G,\delta,x,y)\cap
                F_n)/m(F_n)\\ &\leq
                \varlimsup_{n\to\infty}m(C\dsep(X,\cL,\delta^-_\delta,x,y)\cap
                CF_n'')/m(CF_n')\\ &=
                \varlimsup_{n\to\infty}m(C\dsep(X,\cL,\delta^-_\delta,x,y)\cap
                CF_n'')/m(CF_n'')\cdot
                |F_n''|/|F_n'|\\ &=\ad_{\Fol''}(\dsep(X,\cL,\delta^-_\delta,x,y))=\ad_{\Fol'}(\dsep(X,\cL,\delta^-_\delta,x,y))\\ &=
                \varlimsup_{n\to\infty}m(C\dsep(X,\cL,\delta^-_\delta,x,y)\cap
                CF_n')/m(CF_n'')\\ &\leq
                \varlimsup_{n\to\infty}m(\dsep(X,G,\delta^+_{\delta},x,y)\cap
                F_n)/m(F_n)\\ &=\ad_\Fol(\dsep(X,G,\delta^+_{\delta},x,y)).\qedhere
	\end{align*}
\end{proof}
\begin{remark}
 There is no known general characterisation of groups that allow for uniform lattices.
 However, one well-known consequence of the existence of a lattice in a group $G$ is that $G$ is unimodular (for a definition in the context of amenable groups, see the paragraph before Corollary~\ref{cor: unimodular and minimal then mean equicont iff finite sep numbers}).
 
 Prominent examples of groups with lattices are $\R^d$, or the Heisenberg group and, of course, discrete groups (which, in general, may obviously allow for non-trivial lattices).
 A natural example of an amenable group without lattices are the $p$-adic numbers.
\end{remark}

\begin{remark}\mbox{}
\begin{enumerate}
 \item[(a)]
	If $(F_n)$ is a van Hove sequence, then
	the sets $F_n'$ and $F_n''$ in the above proof are explicitly
	given by $F_n'=\{h\in \cL\;|\;Ch\ssq F_n \}$ and $F_n''=\{h\in \cL\;|\;Ch\cap F_n\neq \emptyset\}$, see the 
	proof of \cite[Lemma~3.2]{Hauser2021}.
 \item[(b)] Let us briefly comment on the necessity of the passage through Proposition~\ref{prop: Folner=van Hove} in the above
	proof.
	As mentioned in Remark~\ref{rem: discrete groups and van Hove}, a Følner sequence in a discrete group is necessarily a van Hove sequence.
	Consequently, given a F\o lner sequence $(F_n')$ in the lattice $\cL$ of $G$,
	$(F_n')$ is actually a van Hove sequence and therefore, one can show that
	$(CF_n')$ defines a van Hove sequence in $G$.
	Accordingly, whenever we seek to bound a Følner sequence $(F_n)$ in $G$ from below and above by sequences $(CF_n')$ and $(CF_n'')$ similarly as in the previous proof, we actually bound $(F_n)$ by van Hove sequences.
	It turns out that this implies that $(F_n)$ itself must be van Hove.
	These observations are straightforward (though slightly tedious) to check.
\end{enumerate}		
\end{remark}


\section{On finiteness of separation numbers}\label{sec: finite vs infinite sep numbers}
This section deals with the scope of amorphic complexity.
In particular, we identify mean equicontinuous systems as those systems where separation numbers
are finite with respect to every F\o lner sequence and amorphic complexity may hence be finite itself.
Moreover, we show that positive entropy as well as weak mixing
imply infinite separation numbers.
\subsection{Mean equicontinuity and finite separation numbers}\label{sec: mean equicont finite sep numbers}

We next discuss a natural class of dynamical systems with finite separation numbers:
the class of mean equicontinuous systems, see \cite{Auslander1959,Robinson1996,HaddadJohnson1997,Robinson1999,Cortez2006,Vorobets2012,
DownarowiczGlasner2016,Glasner2018,LcackaStraszak2018,FuhrmannGroeger2020,FuhrmannKwietniak2020,GroegerLukina2021,FuhrmannGroegerLenz2022} 
for numerous examples.
In our discussion of mean equicontinuity, we follow the terminology of \cite{FuhrmannGroegerLenz2022}.
Given a left or right Følner sequence $\Fol$, a system $(X,G)$ is 
\emph{(Besicovitch) $\Fol$-mean equicontinuous} if for all
$\eps>0$ there is $\delta>0$ such that for all $x,y\in X$ with $d(x,y)<\delta$
we have
\[
	D_{\Fol}(x,y)
	\=\varlimsup\limits_{n\to\infty}1/m(F_n)\int\limits_{F_n} d(tx,ty)\,dm(t)<\eps.
\]
In this case, $D_\Fol$ clearly defines a continuous pseudometric on $X$.
Thus, by identifying points $x,y\in X$ with $D_\Fol(x,y)=0$, we obtain a compact
metric space which we denote by $X/D_\Fol$.

Before we proceed, let us briefly recall the concept of the (upper) box dimension
of a compact metric space $(M,d)$.
Given $\eps>0$, we call a subset $S$ of $M$ \emph{$\eps$-separated}
if for all $s\neq s'\in S$ we have $d(s,s')\geq\eps$ and denote
by $M_\eps$ the maximal cardinality of an $\eps$-separated subset of $M$.
It is well known and easy to see that $M_\eps<\infty$ due to compactness.
With this notation, the \emph{upper box dimension} of $M$ is defined as
\begin{align*}\label{eq: definition box dimension}
	\overline\Dim_B(M)=\varlimsup\limits_{\eps\to 0}
		\frac{\log M_\eps}{-\log\eps}.
\end{align*}

Now, for $\Fol$-mean equicontinuous $(X,G)$, we have
\[
	D_{\Fol}(x,y)
	\geq\varlimsup\limits_{n\to\infty}1/m(F_n)\int\limits_{F_n}
		\mathbf{1}_{[\delta,\infty)}(d(tx,ty))\cdot d(tx,ty)\,dm(t)
	\geq\delta\cdot\ad_{\Fol}(\dsep(X,G,\delta,x,y))
\]
for all $\delta>0$ and $x,y\in X$ and hence, $(X/D_\Fol)_{\delta \nu}\geq \Sep_\Fol(X,G,\delta,\nu)$.
It follows

\begin{proposition}\label{prop: mean equicontinuous implies finite sep}
	If $(X,G)$ is $\Fol$-mean equicontinuous for some left or right
	Følner sequence $\Fol$, then it has finite separation numbers 
	with respect to $\Fol$ and
	\begin{equation*}
		\oac_{\Fol}(X,G)\leq\odim_B(X/D_\Fol).
	\end{equation*}
\end{proposition}

It is important to note that if $\Fol$ is a left Følner sequence, then
$D_\Fol$ is not necessarily invariant.
In particular, the equivalence relation defined by $D_\Fol$ may not define
a factor of $(X,G)$ even if $D_\Fol$ is continuous.
However, it is easy to see that $D_\Fol$ is invariant if $\Fol$ is a
right Følner sequence.
We utilize this observation below.

In any case, it is certainly desirable to have an invariant pseudometric
which does not depend on a particular (right) Følner sequence.
To that end, we may consider
\[
	D(x,y)\=\sup\{D_{\Fol}(x,y)\;|\;\Fol\textnormal{ is a left Følner sequence}\}
\]
which is, in fact, invariant (see \cite[Proposition~3.12]{FuhrmannGroegerLenz2022}).
We say $(X,G)$ is \emph{(Weyl) mean equicontinuous} if $D$ is continuous.
\begin{proposition}[{\cite[Proposition~5.7]{FuhrmannGroegerLenz2022}}]\label{prop: right F mean equicont implies mean equicont}
	Suppose $(X,G)$ is $\Fol$-mean equicontinuous for some right Følner
	sequence $\Fol$.
	Then $(X,G)$ is mean equicontinuous.
\end{proposition}

Given a left or right Følner sequence $\Fol$, a system $(X,G)$ is called
\emph{$\Fol$-mean sensitive} if there exists $\eta>0$ such that for every
open set $U\ssq X$ we can find $x,y\in U$ with $D_\Fol(x,y)\ge \eta$.
Moreover, we say $(X,G)$ is \emph{(Weyl) mean sensitive} if there
exists $\eta>0$ such that for every open set $U\ssq X$ we can find $x,y\in U$
with $D(x,y)\ge \eta$.
We have the following direct generalisation of the equivalence of (1) and
(3) in \cite[Proposition 5.1]{LiTuYe2015} whose proof extends almost literally
to the current setting.

\begin{proposition}
	The system $(X,G)$ is $\Fol$-mean sensitive (with respect to a left
	or right Følner sequence $\Fol$) if and only if there is $\eta>0$ such that
	for every $x\in X$ we have that $\{y\in X\;|\;D_\Fol(x,y)\ge\eta\}$
	is residual in $X$.
\end{proposition}
Clearly, if $\ad_{\Fol}(\dsep(X,G,\eta/2,x,y))<\eta/2$, then $D_\Fol(x,y)\leq\eta/2+(1-\eta/2)\cdot\eta/2<\eta$ (assuming, without loss of generality, that the maximal distance of points in $X$ is $1$).
\begin{corollary}\label{cor: mean sensitive infinite sep}
	If a dynamical system $(X,G)$ is $\Fol$-mean sensitive (for a left 
	or right Følner sequence $\Fol$), then it has infinite separation
	numbers with respect to $\Fol$.
\end{corollary}

In the following, we take a closer look at the relation between mean
equicontinuity and mean sensitivity in the minimal case.
The proof of the next statement is similar to the one for $\Z$-actions 
\cite[Proposition~4.3 \& Theorem~5.4--5.5]{LiTuYe2015}, see also
\cite[Theorem 8]{Garcia-Ramos2017} \& \cite[Theorem 2.7]{Garcia-RamosMarcus2019}
and  \cite[Corollary~5.6]{ZhuHuangLian2022} for similar statements for
abelian (continuous) groups and for countable amenable groups, respectively.
For the convenience of the reader, we provide a direct proof in the current setting.

\begin{lemma}\label{lem: mean equicont vs sensitive dichotomy}
	Let $(X,G)$ be minimal.
	Then $(X,G)$ is either mean equicontinuous or mean sensitive.
	Furthermore, if $(X,G)$ is mean sensitive, then it is $\Fol$-mean sensitive for
	every right Følner sequence $\Fol$.	
\end{lemma}
\begin{proof}
	Suppose $(X,G)$ is not mean equicontinuous.
	That is, there is $x\in X$ and $\eta>0$ such that for all $\delta>0$
	there is $y_\delta\in B_\delta(x)$ with $D(x,y_\delta)>\eta$.
	Now, given any open set $U\ssq X$, there is $g\in G$ and $\delta_0>0$
	such that $g B_{\delta_0}(x)\ssq U$.
	Since $D$ is invariant, we have $D(gx,gy_{\delta_0})=D(x,y_{\delta_0})>\eta$
	which proves the first part.
	
	For the second part, observe that Proposition~\ref{prop: right F mean equicont implies mean equicont}
	gives that for every right Følner sequence $\Fol$ there exist $x\in X$
	and $\eta>0$ such that for all $\delta>0$ there is $y\in B_\delta(x)$
	with $D_\Fol(x,y)>\eta$.
	Since $\Fol$ is assumed to be a right Følner sequence, $D_\Fol$ is
	invariant and we can argue similarly as for $D$ to obtain $\Fol$-mean sensitivity.
\end{proof}

\begin{remark}
	Recall that $G$ acts \emph{effectively} on $X$ if for all $g\in G$ there
	is $x\in X$ such that $gx\neq x$.
	According to \cite[Corollary~7.9]{FuhrmannGroegerLenz2022}, 
	$G$ is necessarily maximally almost periodic (see \cite{FuhrmannGroegerLenz2022}
	and references therein) if $G$ allows for a minimal, mean equicontinuous 
	and effective action on a compact metric space $X$.
	Hence, Lemma~\ref{lem: mean equicont vs sensitive dichotomy} gives that
	every minimal effective action by a group which is not maximally almost
	periodic (such as the \emph{continuous} Heisenberg group $H_3(\R)$)
	is mean sensitive.
\end{remark}

Recall that a locally compact $\sigma$-compact amenable group $G$ is \emph{unimodular}
if and only if $G$ allows for a \emph{two-sided Følner sequence},
that is, a sequence $\Fol$ which is both a left and a right Følner sequence.
In conclusion to the above statements, we obtain
\begin{corollary}\label{cor: unimodular and minimal then mean equicont iff finite sep numbers}
	Suppose $G$ is unimodular and $(X,G)$ is minimal.
	Then $(X,G)$ is mean equicontinuous
	if and only if the separation numbers of $(X,G)$ are finite with respect to
	every left Følner sequence.
\end{corollary}
\begin{proof}
	By definition, mean equicontinuity implies $\Fol$-mean equicontinuity
	with respect to every left Følner sequence.
	Hence, the ``only if''-part follows from Proposition~\ref{prop: mean equicontinuous implies finite sep}.
	
	For the other direction, let $\Fol$ be a two-sided Følner sequence.
	Since we assume the separation numbers with respect to $\Fol$ to be finite,
	we have that $(X,G)$ is not $\Fol$-mean sensitive.
	Since $D_\Fol$ is invariant, we can argue similarly as in 
	Lemma~\ref{lem: mean equicont vs sensitive dichotomy} to obtain that
	$(X,G)$ is $\Fol$-mean equicontinuous.
	Utilizing Proposition~\ref{prop: right F mean equicont implies mean equicont},
	we obtain the desired statement.
\end{proof}


\subsection{Entropy, mixing and infinite separation numbers}\label{sec: entropy mixing infinite sep numbers}

In this section, we discuss how chaotic behaviour---more specifically: weak mixing or positive entropy---implies infinite separation numbers.
Here, we occasionally have to assume that a F\o lner sequence we consider is \emph{tempered}, that is, there is
$C > 0$ such that for all $n$ we have $m(\bigcup_{k<n}F_k^{-1}F_n)<C\cdot m(F_n)$.
It is well known that every F\o lner sequence allows for a tempered subsequence, see 
\cite[Proposition~1.4]{Lindenstrauss2001}.

In line with \cite{GlasnerWeiss2016}, we call an invariant measure $\mu$ of $(X,G)$ \emph{weakly mixing}
if for every system $(Y,G)$ and all of its ergodic measures $\nu$ we have that $\mu\times \nu$ is ergodic for $(X\times Y,G)$.
Hence, if $\mu$ is weakly mixing,
$\mu^m=\bigtimes_{k=1}^m \mu$ is ergodic for $(X^m,G)$ and all $m\in \N$.
\begin{theorem}
	Let $(X,G)$ be a dynamical system with a weakly mixing measure $\mu$ 
	and suppose the support of $\mu$ is not a singleton.
	Then $(X,G)$ has infinite separation numbers with respect to every 
	F\o{}lner sequence.
\end{theorem}
\begin{proof}
 For a tempered F\o{}lner sequence,
 the proof is similar to that of the respective statement for $\Z$-actions (\cite[Theorem~2.2]{FuhrmannGroegerJaeger2016}) if
 we replace Birkhoff's Ergodic Theorem by Lindenstrauss' Pointwise Ergodic Theorem \cite[Theorem~1.2]{Lindenstrauss2001}.
 Here, we have to make use of the ergodicity of $\mu^m$ just as in \cite{FuhrmannGroegerJaeger2016}.
 
 Now, given an arbitrary F\o{}lner sequence, we can always go over to a tempered subsequence 
 (see \cite[Proposition~1.4]{Lindenstrauss2001}).
 This gives infinite separation numbers for a subsequence and hence, due to the $\limsup$ in \eqref{eq: upper density},
 infinite separation numbers for the original sequence.
\end{proof}

We next turn to systems with positive topological entropy.
Our goal is to show
\begin{theorem}\label{thm: positive entropy infinite separation numbers}
 Suppose $G$ allows for a uniform lattice and the dynamical system
 $(X,G)$ has positive topological entropy.
 Then $(X,G)$ has infinite separation numbers with respect to every F\o{}lner sequence in
 $G$.
\end{theorem}
\begin{remark}
 Observe that the proof of a similar statement for $\Z$-actions (see \cite[Theorem~2.3]{FuhrmannGroegerJaeger2016}) 
 utilised results that are only available for $G=\Z$.
 The present approach provides an alternative to the somewhat implicit argument in
 \cite{FuhrmannGroegerJaeger2016}.
\end{remark}
\begin{remark}
	We do not make explicit use of the actual definition of entropy in the
	following and rather utilize results from the theory of topological 
	independence. 
	Therefore, we refrain from discussing the basics of 
	entropy theory in the present work.
	Interested readers are referred to e.g.\ \cite{OrnsteinWeiss1987,KerrLi2016,Bowen2020,Hauser2021,HauserSchneider2022}
	for a background and further references.
\end{remark}

In order to prove Theorem~\ref{thm: positive entropy infinite separation numbers}, we first restrict to actions of countable discrete (and, as throughout assumed, amenable) groups.

\begin{definition}[{cf. \cite[Definition~8.7]{KerrLi2016}}]
 Let $(X,G)$ be a dynamical system and suppose $G$ is countable and discrete. 
 Given a pair $\mathbf{A} = (A_0,A_1)$ of subsets of $X$, we say that a set $J\subseteq G$ 
 is an \emph{independence set} for $\mathbf{A}$ if for
 every non-empty finite subset $I\subseteq J$ and every $(s_g)_{g\in I}\in\{0,1\}^I$ 
 there exists $x\in X$ with $g x\in A_{s_g}$ for every $g\in I$.
\end{definition}

\begin{theorem}[{\cite[Theorem~12.19 \& Proposition~12.7]{KerrLi2016}}]\label{thm: KerrLi: independence set of positive density}
 Suppose $G$ is discrete and countable and $(X,G)$ is a dynamical system.
 If $(X,G)$ has positive topological entropy, then there is a pair 
 $\mathbf{A}=(A_0,A_1)$ of disjoint compact subsets of $X$
 and $d>0$ such that for every tempered 
 F\o{}lner sequence $\Fol=(F_n)$ in $G$ there is an independence set $J$ of
 $\mathbf{A}$ with $\ad_\Fol(J)=\lim_{n\to\infty} |F_n\cap J|/|F_n|\geq d>0$.
\end{theorem}
Let $\mathbf A$, $\Fol$ and $J\ssq G$ be as in the above statement.
Observe that due to the compactness of $A_0$ and $A_1$ we actually have that for every $s=(s_j)_{j\in J}\in \{0,1\}^J$ there exists $x\in X$ which \emph{follows}
$s$, that is, $j x\in A_{s_j}$ for every $j\in J$.

\begin{lemma}\label{lem: positive entropy -> infinite sep numbers in discrete groups}
 Let $G$ be a countable discrete group and suppose
 $(X,G)$ has positive topological entropy.
 Then $(X,G)$ has infinite separation numbers with respect to every F\o{}lner sequence in
 $G$.
 In fact, there are $\delta>0$ and $\nu\in (0,1]$ such that for every F\o lner sequence
 there is an uncountable $(\delta,\nu)$-separated set.
\end{lemma}
\begin{proof}
Let $\mathbf A=(A_0,A_1)$ and $d>0$ be as in Theorem~\ref{thm: KerrLi: independence set of positive density}.
Given a F\o lner sequence $\mc F$, we may assume without loss of generality
that $\mc F$ is tempered.
By Theorem~\ref{thm: KerrLi: independence set of positive density}, we have an associated independence set $J\ssq G$ for $\mathbf A$ 
with $\ad_\Fol(J)\geq d$.
Set $\delta=\dist(A_0,A_1)$ and $\nu =d/2\leq \ad_\Fol(J)/2$.
 Our goal is to show that there is an infinite subset $S \ssq\{0,1\}^J$ such that
 whenever $x,y\in X$ follow distinct elements in $S$, then 
 $\ad_{\Fol}(\dsep(X,G,\delta,x,y))\geq \nu$.
 
 To that end, we first define a sequence $(G_n)_{n\in \N}$ of pairwise disjoint non-empty finite subsets of $G$ such that for every infinite set $M \ssq \N$ we have
\begin{align}\label{eq: density of Gk-s}
    \ad_{\Fol}(\bigcup_{n\in M} G_n)\ge  1-\nu.
\end{align}
 We may do so by starting with $G_1=F_1$.
 Assuming we have already chosen $G_1,\ldots,G_n$ for some $n\in \N$, 
 let $N=N(n)\in \N$ be large enough to guarantee that
  \[
  |F_N\setminus(G_1\cup\ldots\cup G_n)| \ge (1-\nu)|F_N|
  \]
 and set $G_{n+1}=F_N\setminus(G_1\cup\ldots\cup G_n)$.
 Note that this gives that $(G_n)$ satisfies \eqref{eq: density of Gk-s} for every 
 infinite $M\subseteq\N$ because
 \[\ad_{\Fol}(\bigcup_{n\in M} G_n)\ge \varlimsup_{\substack{n\to\infty \\n\in M}}\frac{|F_{N(n-1)}\cap G_n|}{|F_{N(n-1)}|}\ge 1-\nu,\]
 for any infinite $M\ssq \N$.
 
 
 Now, let $E$ be an uncountable family of subsets of $\N$ such that $M\triangle M'$ is infinite for distinct $M,M'\in E$. 
Given $M\in E$, we define $s^M\in \{0,1\}^J$ by
\[
s^M_j=\begin{cases}1 &\text{ if }j\in G_n\text{ and }n\in M,\\
0 &\text{ otherwise.}\end{cases}
\]
Set
$S=\{s^M\in\{0,1\}^J\;|\; M\in E\}$.
Given $s\in S$, choose some $x(s)\in X$ which follows $s$ (recall the discussion before the statement).
It is straightforward to see that for distinct $M,M'\in E$, we have
for $x=x(s^M)$ and $x'=x(s^{M'})$ that
\begin{align*}
\dsep(X,G,\delta,x,x') &=
\{ g\in G\;|\;d(gx,gx') \ge \delta \}\supseteq\{g\in J\;|\;\;s^M_g\neq s^{M'}_g\}
\\
&= J\cap \big(\bigcup_{n\in M\triangle M'}G_n\big).
\end{align*}
Using \eqref{eq: density of Gk-s}, we obtain
\[
\ad_\Fol\big(J\cap \bigcup_{n\in M\triangle M'}G_n\big)\geq \ad_\Fol(J)/2\geq\nu.
\]
Hence, $\{x(s)\in X\;|\;s\in S\}$ is the uncountable $(\delta,\nu)$-separated set we sought.
\end{proof}

\begin{remark}
	The above arguments are heavily based on the concept of topological 
	independence.
	Another notion in the regime of zero entropy in which related ideas
	feature prominently is \emph{topological sequence entropy}, see 
	for instance \cite{Goodman1974,KerrLi2007,Canovas2008,HuangYe2009,SnohaYeZhang2020}.
	An in-depth comparison of topological sequence entropy and amorphic 
	complexity---and possibly other so-called slow entropy notions (see
	for example \cite{Carvalho1997,HasselblattKatok2002,DouHuangPark2011,
	Marco2013, KongChen2014, CantatParisRomaskevich2021})---is certainly an interesting
	topic for further investigation.
	As this is beyond the scope of the present work, we refer the interested
	reader to \cite{FuhrmannGroegerJaeger2016,GroegerJaeger2016} for first steps
	in this direction.
\end{remark}

\begin{proof}[Proof of Theorem~\ref{thm: positive entropy infinite separation numbers}]
 Let us denote by $\cL$ a lattice (as provided by the assumptions) in $G$.
 Note that since $G$ is $\sigma$-compact, we have that $\cL$ is countable.
 
 Due to \cite[Theorem~5.2]{Hauser2021}, positive topological entropy of $(X,G)$ implies positive topological entropy of $(X,\cL)$.
 Hence, Lemma~\ref{lem: positive entropy -> infinite sep numbers in discrete groups} gives
 that $(X,\cL)$ has infinite separation numbers with respect to every F\o{}lner sequence.
 Due to Lemma~\ref{lem: ac of group and lattice coincide}, this implies infinite separation numbers of $(X,G)$ with respect to every F\o{}lner sequence.
\end{proof}

We close this section with two immediate corollaries of the above.
The first one is a consequence of Theorem~\ref{thm: positive entropy infinite separation numbers},
Proposition \ref{prop: mean equicontinuous implies finite sep} and Lemma \ref{lem: mean equicont vs sensitive dichotomy}.
For the prize of additionally assuming minimality, it extends \cite[Theorem~6.6]{ZhuHuangLian2022} to actions of non-abelian and/or uncountable groups.

\begin{corollary}
	If $G$ allows for a uniform lattice and $(X,G)$ is minimal and has
	positive topological entropy, then $(X,G)$ is (Weyl) mean sensitive.
\end{corollary}

The second consequence extends \cite[Theorem~1.2]{ZhuHuangLian2022} and follows from Proposition~\ref{prop: mean equicontinuous implies finite sep} and Theorem~\ref{thm: positive entropy infinite separation numbers}.\footnote{Observe that \cite[Theorem~1.2]{ZhuHuangLian2022} is formulated in terms of Banach mean equicontinuity which, however, in the setting of \cite{ZhuHuangLian2022} and beyond, coincides with the notion of mean equicontinuity discussed here, see \cite[Theorem~12]{LcackaStraszak2018} as well as \cite[Theorem~4.3]{ZhuHuangLian2022}.}
\begin{corollary}
	If $G$ allows for a uniform lattice and $(X,G)$ is mean equicontinuous, then $(X,G)$ has zero topological entropy.
\end{corollary}
It is worth mentioning that under the assumption of an appropriate variational principle, this statement also follows from the results in \cite[Theorem~3.15]{FuhrmannGroegerLenz2022}.
Alternatively, in the setting of discrete amenable groups, we could also use that discrete spectrum is equivalent to measure-theoretic nullness, see \cite{KerrLi2009} for more details.

\begin{remark}
 Elaborating on the relation of the results in this section and \cite{ZhuHuangLian2022}, we would like to mention that for minimal actions of countable abelian groups, Theorem~\ref{thm: positive entropy infinite separation numbers} readily follows from \cite[Theorem~6.6]{ZhuHuangLian2022} in conjunction with Corollary~\ref{cor: unimodular and minimal then mean equicont iff finite sep numbers}.
\end{remark}


\section{Independence of Følner sequences}\label{sec: independence Folner seq}

In general, amorphic complexity might depend on the particular F\o{}lner sequence with respect
to which we compute the separation numbers.
For $G=\Z$, this can be seen by considering the example in
 \cite[page 541]{FuhrmannGroegerJaeger2016}.
 There, $\ac_\Fol(X,\Z)=\infty$ for $\Fol=([0,n))_{n\in \N}$ while
 $\ac_{\Fol'}(X,\Z)=0$ for $\Fol'=((-n,0])_{n\in \N}$.
 
The goal of this section is to show
\begin{theorem}\label{thm: ac independent of Folner}
 Let $(X,G)$ be a dynamical system whose product $(X^2,G)$ is pointwise uniquely ergodic.
 Then $\oac_\Fol(X,G)$ and $\uac_\Fol(X,G)$ are independent of the particular (left) F\o{}lner
 sequence $\Fol$.
\end{theorem}
\begin{remark}\label{rem: mean equicontinuous implies pointwise uniquely ergodic product}
 Notice that due to \cite[Theorem~1.2]{FuhrmannGroegerLenz2022}, the above gives that amorphic complexity
 of mean equicontinuous systems is independent of the particular F\o{}lner sequence.
\end{remark}

In fact, we have the following stronger statement which
immediately yields Theorem~\ref{thm: ac independent of Folner}.

\begin{theorem}\label{theorem:independence Folner sequences sep numbers and ac}
	Let $(X,G)$ be a dynamical system whose product
	$(X^2,G)$ is pointwise uniquely ergodic.
	The following holds.
	\begin{enumerate}
		\item[(i)] Suppose there is a Følner sequence $\Fol$ such that $\Sep_{\Fol}(X,G,\delta,\nu)=\infty$ 
			for some $\delta,\nu\in (0,1)$.
			Then there exists $\delta_0>0$ such that $\Sep_{\Fol'}(X,G,\delta',\nu)=\infty$ for
			every Følner sequence $\Fol'$ and every $\delta'\in(0,\delta_0]$.
		\item[(ii)]
		Let $\Fol^0$ and $\Fol^1$ be F\o lner sequences and suppose
		$\Sep_{\Fol^0}(X,G,\delta,\nu)<\infty$ for all $\nu,\delta\in (0,1)$.
		Then there is a cocountable set $A\in (0,1)$ such that for all $\delta\in A$ we have
		$\Sep_{\Fol^0}(X,G,\delta,\nu)=\Sep_{\Fol^1}(X,G,\delta,\nu)$ for all but countably many $\nu$.
	\end{enumerate}
\end{theorem}
\begin{proof}
	Without loss of generality, we may assume that $\diam(X)=1$.
	We start by providing some general observations.
	Given $\delta\in(0,1)$, let $(h_\ell)$ and $(H_\ell)$ be sequences of non-decreasing continuous self-maps on $[0,1]$.
	For large enough $\ell\in\N$, assume that $h_\ell(z)=0$ for $z\in [0,\delta]$ and $h_\ell(z)=1$ for
	$z\in [\delta+1/\ell,1]$ as well as $H_{\ell}=0$ on $[0,\delta-1/\ell]$ and
	$H_\ell=1$ on $[\delta,1]$.
	Clearly, $h_{\ell}(z)\leq\mathbf{1}_{[\delta,1]}(z)\leq H_\ell(z)$
	for all $z\in [0,1]$ and large enough $\ell\in\N$.
	Hence, for all $x,y\in X$, every Følner sequence $\Fol=(F_n)$, and sufficiently large $\ell$, we have
	\begin{align}\label{eq: sandwich}
	\begin{split}
		&\int\limits_{X^2}h_{\ell}(d(v,w))d\mu_{(x,y)}(v,w)=
		\lim_{n\to\infty}1/|F_n|\cdot \int\limits_{F_n}h_{\ell}(d(sx,sy))dm(s)\\
		&\leq \lim_{n\to\infty}1/|F_n|\cdot \int\limits_{F_n} \mathbf{1}_{[\delta,1]}(d(sx,sy))dm(s)
		=\ad_{\Fol}(\dsep(X,G,\delta,x,y))\\
		&\leq \lim_{n\to\infty} 1/|F_n|\cdot \int\limits_{F_n}H_{\ell}(d(sx,sy))dm(s)=\int\limits_{X^2}H_{\ell}(d(v,w))d\mu_{(x,y)}(v,w),
	\end{split}
	\end{align}
	where we used the pointwise unique ergodicity
	of $(X^2,G)$ and where $\mu_{(x,y)}$ denotes the unique invariant measure on the orbit closure of $(x,y)\in X^2$.
	Sending $\ell\to \infty$, we obtain equality in \eqref{eq: sandwich}
	unless
	\begin{equation}\label{eq:bad deltas for approximation}
		\mu_{(x,y)}(\{(v,w)\in X^2\;|\;d(v,w)=\delta\})>0.
	\end{equation}
	In other words, if \eqref{eq:bad deltas for approximation} does not hold, then $\ad_{\Fol}(\dsep(X,G,\delta,x,y))$ 
	is actually independent of the Følner sequence $\Fol$.
	Notice that given $(x,y)$, there can be at most countably many $\delta$ which verify \eqref{eq:bad deltas for approximation}.
	
	Let us prove statement (i).
	Suppose $\Fol$ is a Følner sequence and $\Sep_{\Fol}(X,G,\delta,\nu)=\infty$ for some $\delta,\nu\in(0,1)$.
	Let $\mathcal S$ be a countable family of finite $(X,G,\delta,\nu)$-separated sets (with respect to $\Fol$)
	such that $\sup_{S\in\mathcal S}\#S=\infty$.
	Further, let $C\ssq(0,1)$ be the collection of all $\delta\in (0,1)$ such that for some $S\in\mathcal S$
	there are $(x,y)\in S^2$ such that \eqref{eq:bad deltas for approximation} holds.
	As $C$ is at most countable, there exists $\delta_0\in(0,\delta]$ such that for any $S\in\mathcal S$ we have
	\begin{equation*}
		\ad_{\Fol'}(\dsep(X,G,\delta_0,x,y))=\ad_{\Fol}(\dsep(X,G,\delta_0,x,y))\geq\ad_{\Fol}(\dsep(X,G,\delta,x,y))\geq\nu
	\end{equation*}
	for all $x\neq y\in S$ and any Følner sequence $\Fol'$ where we used that $\abs{\dsep(X,G,\cdot,x,y)}$ is non-increasing.
	It straightforwardly follows that each $S$ is $(X,G,\delta',\nu)$-separated with respect to any F\o lner sequence $\Fol$ 
	and every $\delta'\in(0,\delta_0]$.
	As $S$ can be chosen arbitrarily large, this proves the first assertion.
		
	Let us consider (ii).
	First, observe that due to (i), we have $\Sep_{\Fol^1}(X,G,\delta,\nu)<\infty$ for all $\delta,\nu \in (0,1)$.
	Given $\delta\in(0,1)$, we call $\nu\in(0,1)$ \emph{$\delta$-singular} if
	$\Sep_{\Fol^i}(X,G,\delta,\nu)<\Sep_{\Fol^i}(X,G,\delta-\eps,\nu)$ for all $\eps>0$ and some $i\in\{0,1\}$.
	Otherwise, we say $\nu$ is \emph{$\delta$-regular}.
	The collection of all $\delta$-singular elements of $(0,1)$ is denoted by $B_\delta$.
	We say $\delta$ is \emph{singular} if $B_\delta$ is uncountable.
	Otherwise, we call $\delta\in(0,1)$ \emph{regular}. 
	The collection of all singular $\delta$ in $(0,1)$ is denoted by $B$.
	We set $A=(0,1)\setminus B$.
	
	Next, we show that for all $\delta\in(0,1)$ and each $\nu\in B_\delta^c$ we have
	$\Sep_{\Fol^0}(X,G,\delta,\nu)=\Sep_{\Fol^1}(X,G,\delta,\nu)$.
	To prove (ii), it then remains to show that $B$ is countable.

	Given $\delta\in(0,1)$, let $\nu\in(0,1)$ be $\delta$-regular.
	By definition, there is $\eps>0$ such that
	$\Sep_{\Fol^i}(X,G,\delta,\nu)=\Sep_{\Fol^i}(X,G,\delta',\nu)$ for all $\delta'\in (\delta-\eps,\delta)$ and $i=0,1$.
	Let $S\ssq X$ be $\delta$-$\nu$-separated w.r.t. $\Fol^0$ and suppose $S$ is of maximal cardinality.
	As $S$ is finite, the collection of all $\delta\in (0,1)$ which verify \eqref{eq:bad deltas for approximation}
	for some pair $(x,y)\in S^2$ is countable.
	There is hence $\delta'\in (\delta-\eps,\delta)$ which does not verify \eqref{eq:bad deltas for approximation}
	for any $(x,y)\in S^2$.
	Clearly, $S$ is $\delta'$-$\nu$-separated for $\Fol^0$.
	By the above, $S$ is also $\delta'$-$\nu$-separated for $\Fol^1$.
	Hence,
	\[
	 \Sep_{\Fol^1}(X,G,\delta,\nu)=\Sep_{\Fol^1}(X,G,\delta',\nu)\geq \Sep_{\Fol^0}(X,G,\delta',\nu)=\Sep_{\Fol^0}(X,G,\delta,\nu).
	\]
	By changing the roles of $\Fol^0$ and $\Fol^1$, we obtain the converse inequality and 
	accordingly $\Sep_{\Fol^0}(X,G,\delta,\nu)=\Sep_{\Fol^1}(X,G,\delta,\nu)$ for all $\delta$-regular $\nu$.
	
	It remains to show that $B$ is countable.
	To that end, we need the following
	\begin{claim}
	 If $\delta\in(0,1)$ is singular, then $B_\delta$ has non-empty interior.
	\end{claim}
	\proof[Proof of the claim]
	Let $\nu\in (0,1)$ be $\delta$-singular and $\nu'\in(0,\nu)$ be $\delta$-regular.
	Observe that due to the monotonicity in both arguments of $\Sep_{\Fol^i}(X,G,\cdot,\cdot)$,
	there has to be a \emph{jump point} $\nu_0$ between
	$\nu$ and $\nu'$ (possibly coinciding with $\nu$ or $\nu'$), 
	i.e., a point $\nu_0$ such that for $i=0$ or $i=1$ we have
	$\Sep_{\Fol^i}(X,G,\delta,\nu_0-\eps)>\Sep_{\Fol^i}(X,G,\delta,\nu_0)$ for all $\eps>0$.
	As $\Sep_{\Fol^i}(X,G,\delta,\cdot)$ is non-increasing and integer-valued, 
	each compact subinterval of $(0,1)$ can contain at most finitely many
	such jump points.
	Therefore, the set of $\delta$-singular points is a union of isolated points and intervals.
	Since a subset of $(0,1)$ with only isolated points is at most countable,
	this proves the claim.
	\roundqed\smallskip
	
	Now, for a contradiction, assume that $B$ is uncountable.
	By the above claim, $B_\delta$ contains an interval $I_\delta$
	whenever $\delta\in B$.
	Thus, there must be an uncountable set $B'\ssq B$ with 
	$\bigcap_{\delta\in B'} I_\delta\neq \emptyset$.
	Accordingly, there is $\nu\in (0,1)$
	such that $\nu$ is $\delta$-singular for all $\delta\in B'$.
	As $\Sep_{\Fol^i}(X,G,\cdot,\nu)$ is non-increasing, there can be at most countably many $\delta$ with 
	$\Sep_{\Fol^i}(X,G,\delta-\eps,\nu)>\Sep_{\Fol^i}(X,G,\delta,\nu)$ for all $\eps>0$. 
	This contradicts the uncountability of $B'$.
	Hence, $B$ is at most countable.
	This finishes the proof.
\end{proof}


\section{Application to regular model sets}\label{sec: ac and regular model sets}

In this section, we study amorphic complexity of (the dynamical hull
of) model sets.
Given a model set, our third main result provides 
an upper bound for its amorphic complexity which may be understood as a measure of its amorphicity.
We start by collecting a number of preliminary facts concerning Delone sets, cut and project schemes
and their associated dynamics.

\subsection{Delone dynamical systems and model sets}

From now on, in what follows, $G$ is a locally compact second countable
abelian group with Haar measure $m_G$.  
Further, in all of the following, we switch to additive notation for the
group operation in $G$, accounting for its commutativity.
By the Birkhoff-Kakutani Theorem, $G$ is
metrizable and the metric $d_G$ can be chosen to be invariant under $G$.
In fact, open balls with respect to $d_G$ are relatively compact \cite{Struble1974} so that
$G$ is automatically $\sigma$-compact.

A set $\Gamma\ssq G$ is called {\em $r$-uniformly discrete} if there exists
$r>0$ such that $d_G(g,g')>r$ for all $g\neq g'\in\Gamma$.
Moreover, $\Gamma$ is called {\em $R$-relatively dense} (or {\em $R$-syndetic})
if there exists $R>0$ such that $\Gamma\cap B_G(g,R)\neq \emptyset$ for all
$g \in G$, where $B_G(g,R)$ denotes the open $d_G$-ball of radius $R$ centred at $g$.
We call $\Gamma$ a {\em Delone set} if it is uniformly discrete and relatively dense.
The collection of all Delone sets in $G$ will be denoted by $\cD(G)$.

Given $\rho>0$ and $g\in\Gamma$, the tuple $(B_G(0,\rho)\cap (\Gamma-g),\rho)$ is
called a {\em ($\rho$-)patch} of $\Gamma$. 
The set of all patches of $\Gamma$ is denoted by $\cP(\Gamma)$.
A Delone set $\Gamma$ is said to have {\em finite local complexity} (FLC)
if for all $\rho>0$ the number of its $\rho$-patches is finite.
For $\Gamma, \Gamma' \in \cD(G)$, set
\[
	\dist(\Gamma, \Gamma') = \inf \left\{ \varepsilon > 0\;|\;\exists g
	\in B_G(0,\varepsilon) : (\Gamma-g) \cap B_G(0,1/\varepsilon) = \Gamma' \cap
	B_G(0,1/\varepsilon) \right\}.
\]
Then $d(\Gamma,\Gamma') = \min \{ 1/\sqrt{2},\dist(\Gamma,\Gamma')\}$ defines
a metric on $\cD(G)$ (see \cite[Section~2]{LeeMoodySolomyak2002}). 
Moreover, for any Delone set $\Gamma\ssq G$ with FLC the {\em dynamical hull}
of $\Gamma$, defined as
\[
	\Omega(\Gamma) = \overline{\left\{\Gamma-g\;|\;g \in G\right\}},
\]
where the closure is taken with respect to $d$, is compact 
\cite[Proposition~2.3]{Schlottmann1999}.
The dynamical system $(\Omega(\Gamma),G)$, given by the translation action of $G$ on
the hull $\Omega(\Gamma)$, is called a {\em Delone dynamical system}.

The method of choice to construct Delone sets is to utilize a \emph{cut and project scheme} (CPS).
A CPS consists of a triple $(G,H,\cL)$ of two
locally compact abelian groups $G$ ({\em external group}) and $H$ ({\em
internal group}) and a uniform lattice $\cL \ssq G \times H$ 
which is {\em irrational}, that is,
the natural projections $\pi_G : G \times H \to G$ and
$\pi_H : G \times H\to H$ satisfy
\begin{itemize}
	\item[(i)] the restriction $\pi_G \vert_\cL$ is injective;
	\item[(ii)] the image $\pi_H(\cL)$ is dense.
\end{itemize}
If not stated otherwise, we throughout assume that $G$ and $H$ are second countable.
As a consequence of (i), if we let $L=\pi_G(\cL)$ and $L^* = \pi_H(\cL)$,
the {\em star map} 
\[
	* : L \rightarrow L^* : l \mapsto l^* =
	\pi_H\circ \left.\pi_{G}\right|_\cL^{-1}(l)
\]
is well defined and surjective.
Given a precompact set $W \ssq H$ (referred to as {\em window}), we define
the point set
\begin{equation*} 
	\oplam(W)=\pi_G\left(\cL \cap (G\times W)\right)=\{l\in L\;|\;l^*\in W\}.
\end{equation*}
If $W$ is compact and {\em proper} (that is, $\overline{\inte(W)}=W$),
then $\oplam(W)$ is a Delone set and has FLC \cite{Robinson2007}.
In this case, we call $\oplam(W)$ a {\em model set}.
If further $m_H(\partial W)=0$, then we call the window, as well as the
resulting model set, {\em regular}.
Otherwise, we refer to $W$ and $\oplam(W)$ as {\em irregular}.
Delone dynamical systems associated to regular model sets are
mean equicontinuous, see \cite[Theorem 9]{BaakeLenzMoody2007} together with 
\cite[Theorem 1.1]{FuhrmannGroegerLenz2022} as well as 
\cite[Remark~6.2 \& Corollary~6.3]{FuhrmannGroegerLenz2022}.
Note that the converse is not true, see \cite{DownarowiczGlasner2016,FuhrmannGlasnerJaegerOertel2021} and also \cite[Chapter~7.2]{FuhrmannGroegerLenz2022} for a thorough discussion and further references.

We say that a subset $A \ssq H$ is {\em irredundant} if $ \{ h \in H\;|\;h+A=A \} = \{0\}$.
Clearly, if $\partial W$ is irredundant, then so is $W$.
A CPS is called {\em Euclidean} if $G =\R^N$ and $H = \R^M$ for some 
$M,N\in\N$, and {\em planar} if $N=M=1$.
Note that in the Euclidean case, any compact window is irredundant.
Further, observe that if $W$ is not irredundant, it is possible to construct a CPS $(G,H',\cL')$ with irredundant window $W' \ssq H'$ such that for each $\Lambda \in \Omega(\oplam(W))$ with 
$\oplam(\interior(W)) \ssq \Lambda \ssq \oplam(W)$ we have $\oplam(\interior(W')) \ssq \Lambda \ssq \oplam(W')$ (compare \cite[Section~5]{LeeMoody2006} and \cite[Lemma~7]{BaakeLenzMoody2007}).

As $\cL$ is a uniform lattice in $G \times H$, the quotient $\T\=(G \times H) /
\cL$ is a compact abelian group. 
A natural action of $G$ on $\T$ is given by $(u,[s,t]_\cL) \mapsto [s+u,t]_\cL$.
Here, $[s,t]_\cL$ denotes the equivalence class of $(s,t) \in G \times H$ in $\T$. 
Observe that due to the assumptions on $(G,H,\cL)$, this action is equicontinuous,
minimal and has hence a unique invariant measure $\mu_{\T}$.
Furthermore, if $W\ssq H$ is irredundant, $(\T,G)$ is the maximal equicontinuous
factor of $(\Omega(\oplam(W)),G)$ \cite{BaakeLenzMoody2007}.
The respective factor map $\beta$ is also referred to as {\em torus
parametrization}.

Given an irredundant window $W$, the fibres of the torus parametrization
are characterized as follows: for $\Gamma\in\Omega(\oplam(W))$, we have
\begin{equation}\label{eq:flowMorphism}
	\Gamma\in\beta^{-1}([s,t]_\cL) \quad \equi \quad \oplam(\inte(W)+t)-s \ssq
	\Gamma\ssq \oplam(W+t)-s
\end{equation} 
as well as
\begin{equation*}
	\Gamma \in \beta^{-1}([0,t]_\cL)\quad\equi\quad \exists\, (t_j) \in
	{L^*}^\N \text{ with } \jLim t_j = t \text{ and } \jLim \oplam(W+t_j) =
	\Gamma.
\end{equation*}

In the following, we denote by $\vol(\cL)$ the volume of a fundamental domain of $\cL$.
Note that $\vol(\cL)$ is well defined.

\begin{proposition}[{\cite[Proposition 3.4]{HuckRichard2015}}]\label{prop: van Hove densities}
  Let $(G,H,\cL)$ be a CPS and $W\ssq H$ be precompact. 
  Then for every
  van Hove sequence $\Fol=(F_n)$ in $G$ we have
  \begin{equation*}
	\frac{m_H(\inte(W))}{\vol(\cL)}\leq
	\varliminf_{n\to \infty}\frac{\sharp(\oplam(W) \cap F_n)}{m_G(F_n)}\leq
	\varlimsup_{n\to \infty}\frac{\sharp(\oplam (W) \cap F_n)}{m_G(F_n)}\leq
    \frac{m_H(W)}{\vol (\cL)}.
  \end{equation*}
\end{proposition}

Let us collect three more statements which follow easily from
the definition of the metric $d$ on $\cD(G)$.
Similarly to the notion of $(\delta,\nu)$-separation of elements of a dynamical system (see Section~\ref{sec: Introduction}), given a van Hove sequence $\Fol$ in $G$, we set
\[
	\nu_{\Fol}(\delta,\Gamma,\Gamma')
		=\ad_{\Fol}(\{g\in G\;|\;d(g\Gamma,g\Gamma')\geq\delta\}),
\]
where $\delta>0$ and $\Gamma,\Gamma'\in\cD(G)$. 

\begin{proposition}
	For every van Hove sequence $\Fol=(F_n)$ in $G$ we have
	\begin{equation*}
		\nu_{\Fol}(\delta,\Gamma,\Gamma')\leq m_G(B_G(0,1/\delta))
		\varlimsup_{n\to\infty}\frac{\sharp((\Gamma \Delta \Gamma')\cap F_n)}{m_G(F_n)},
		\end{equation*}
	with $\delta>0$ and $\Gamma,\Gamma'\in\cD(G)$. 
\end{proposition}

Accordingly, together with Proposition \ref{prop: van Hove densities}, we get

\begin{corollary} \label{c.no_separation}
  If $m_H(\partial W)=0$ and $\oplam(\inte(W))\ssq \Gamma\ssq \oplam(W)$,
  then $\nu_{\Fol}(\delta,\Gamma,\Gamma')=\nu_{\Fol}(\delta,\oplam(W),\Gamma')$ for all 
  van Hove sequences $\Fol$, $\delta>0$ and $\Gamma'\in \cD(G)$.
\end{corollary}

Finally, observe that

\begin{proposition}\label{prop: distance_of_shifted_delonesets}
  Suppose $\delta>0$, $\Gamma,\Gamma'\in\cD(G)$ and $g\in B_G(0,{\delta/2})$.
  If $d(\Gamma,\Gamma')\geq \delta$, then $d(\Gamma,\Gamma'+g)\geq \delta/2$.
\end{proposition}

\subsection{Upper bound on the amorphic complexity of regular model sets}
We next come to our third main result.
First, recall that for a locally compact $\sigma$-compact group $H$, the upper box dimension
is given by
\[
	\uboxdim(H)=\varlimsup\limits_{\eps\to 0}\frac{\log m_H\big(\overline{B_H(h,\eps)}\big)}{\log\eps},
\]
where $h\in H$ is arbitrary.
Observe that $\uboxdim(H)$ is well defined because of the invariance of the metric
$d_H$ and the Haar measure $m_H$.
Note further that the above definition, as well as the definition of the (upper) box dimension
of compact sets in Section~\ref{sec: mean equicont finite sep numbers}, are special cases
of a more general concept of box dimension.
We refrain from reproducing the slightly technical (and standard) general definition here and refer the interested reader to \cite[Section 1.4]{Edgar1998} instead.

We will also make use of \emph{Minkowski's characterisation}
of the box dimension of a given compact set $M\subseteq H$ by
\[
	\uboxdim(M)=\uboxdim(H)-\varliminf\limits_{\eps\to 0}\frac{\log m_H\big(\overline{B_H(M,\eps)}\big)}{\log\eps}.
\]
The proof of this fact in our setting is similar to the one in the Euclidean
space, see for instance \cite[Proposition 3.2]{Falconer2003}.

Finally, in order to derive upper bounds on amorphic complexity, it is 
convenient to make use of an alternative characterisation which utilises
spanning sets instead of separating sets---similar as in the derivation of upper bounds for topological entropy (or box dimension).
Given $\delta >0$ and $\nu\in (0,1]$, we say a subset $S\subseteq X$ is \emph{$(\delta,\nu)$-spanning} with respect to a Følner sequence $\Fol$ if for all $x\in X$ there exists 
$s\in S$ such that $\ad_{\Fol}(\dsep(X,G,\delta,x,s))<\nu$.  
We denote by $\Span_{\Fol}(X,G,\delta,\nu)$ the smallest cardinality
among the $(\delta,\nu)$-spanning sets with respect to $\Fol$.
It is not difficult to see that
$\Span_{\Fol}(X,G,\delta,\nu)$ instead of $\Sep_{\Fol}(X,G,\delta,\nu)$
can equivalently be used in defining amorphic complexity, see also \cite[Lemma 3.1 \& Corollary 3.2]{FuhrmannGroegerJaeger2016}.

\begin{theorem} \label{t.ac_bound_for_cps}
  Suppose $(G,H,\cL)$ is a cut and project scheme, where $G$ and $H$ are locally
  compact second countable abelian groups.
  Furthermore, let $W\ssq H$ be compact, proper, regular and irredundant and assume that $\uboxdim(H)$ is finite. 
  Then
  \begin{equation}\label{eq: upper bound ac quasicrystal}
    \oac_{\Fol}(\Omega(\oplam(W)),G) \ \leq \ \frac{\uboxdim(H)}{\uboxdim(H)-\uboxdim(\partial W)},
  \end{equation}
  for any F\o lner sequence $\Fol$.
\end{theorem}
\proof 
As $W$ is regular and hence $(\Omega(\oplam(W)),G)$ mean equicontinuous,
we may assume without loss of generality that $\Fol$ is van Hove, see Remark~\ref{rem: mean equicontinuous implies pointwise uniquely ergodic product}
and Theorem~\ref{theorem:independence Folner sequences sep numbers and ac}.
We first choose compact sets $A\ssq G$ and $B\ssq H$ such that $W\ssq B$
and $\pi(A\times B)=\T$, where $\pi:G\times H\to\T=(G\times H)/\cL$ is the canonical
projection. 

Given $(g,h)\in A\times B$, let $\hat\Gamma_{g,h}=\oplam(W+h)-g$.
Observe that $\hat \Gamma_{g,h}$ may not be an element of $\Omega(\oplam(W))$.
While for our asymptotic estimates this will be of no problem (due to 
Corollary~\ref{c.no_separation}), its explicit definition makes 
$\hat \Gamma_{g,h}$ more convenient to deal with in computations.
\begin{claim}
  Let $\delta>0$. 
  If $d_G(g,g')\leq \delta/2$ and
  $d(\hat\Gamma_{g,h},\hat\Gamma_{g',h'})\geq \delta$, then
  \[
	[-g,-h]_\cL\in\pi\left(B_G(0,2/\delta)\times (W\Delta (W+h'-h))\right)=:D(\delta,h'-h).
  \]
\end{claim}
\proof[Proof of the claim]
By Proposition \ref{prop: distance_of_shifted_delonesets}, we know that
$d(\hat\Gamma_{g,h},\hat\Gamma_{g,h'})\geq \delta/2$.
Hence, there exists $(\ell,\ell^*)\in\cL$ with $\ell\in B_G(g,2/\delta)$
and $\ell\in\hat\Gamma_{g,h}\Delta\hat\Gamma_{g,h'}$.
The latter implies that $\ell^*\in (W+h)\Delta(W+h')$.

Equivalently, this means that $\ell-g\in B_G(0,2/\delta)$ and
$\ell^*-h\in W\Delta(W+h'-h)$, so that
\[
	[-g,-h]_\cL=[\ell-g,\ell^*-h]_\cL\in\pi\left(B_G(0,2/\delta)\times
	W\Delta(W+h'-h)\right).
\]
This proves the claim. \roundqed\smallskip

We can now apply the claim to estimate the separation frequency of a pair
$\hat\Gamma_{g,h}$ and $\hat\Gamma_{g',h'}$.
\begin{eqnarray*}
  \nu_{\Fol}(\delta,\hat\Gamma_{g,h},\hat\Gamma_{g',h'}) & = &
  \varlimsup_{n\to\infty}\frac{1}{m_G(F_n)} \int_{F_n}
  \ind_{[\delta,\infty)}(d(\hat\Gamma_{g,h}-t,\hat\Gamma_{g',h'}-t))dt\\ & = &
    \varlimsup_{n\to\infty}\frac{1}{m_G(F_n)} \int_{F_n}
    \ind_{[\delta,\infty)}(d(\hat\Gamma_{g+t,h},\hat\Gamma_{g' +t,h'}))dt\\ & \leq &
      \varlimsup_{n\to\infty}\frac{1}{m_G(F_n)} \int_{F_n}
      \ind_{D(\delta,h'-h)}([-g-t,h]_\cL)dt\\ & \stackrel{(*)}{=} &
      \mu_\T(D(\delta,h'-h)) \\ & \leq & m_G(B_G(0,2/\delta))\cdot
      m_H(W\Delta(W+h'-h)) \\ & \leq & m_G(B_G(0,2/\delta)) \cdot
      m_H(\overline{B_H(\partial W,d(0,h'-h))}),
\end{eqnarray*}
where the equality $(*)$ follows from the unique ergodicity of $(\T,G)$
and the fact that $\mu_\T(\partial D(\delta,h'-h))=0$.

Now, suppose that $\delta>0$ and $\nu>0$ are given. Let
\[
	\eps=\inf\left\{\eta>0\;|\;m_H\left(B_H\left(\partial W,\eta\right)\right)\geq
		\nu/m_G\left(B_G(0,2/\delta)\right) \right\}.
\]
Then we have $m_H(B_H(\partial W,\eps))\leq \nu/m_G(B_G(0,2/\delta))$ but at the same time
$m_H\big(\overline{B_H(\partial W,\eps)}\big)\geq \nu/m_G(B_G(0,2/\delta))$
due to the regularity of Haar measure.
Consequently, if $d_G(g,g')<\delta/2$ and $d_H(h,h')<\eps$, then the first
inequality combined with the above estimate yields that $\hat\Gamma_{g,h}$
and $\hat\Gamma_{g',h'}$ cannot be $(\delta,\nu)$-separated.

For $g\in G$ and $h\in H$, let $\Gamma_{g,h}$ denote some element
of $\Omega(\oplam(W))$ with 
$\oplam(\inte(W)+h)-g\ssq\Gamma_{g,h}\ssq\hat\Gamma_{g,h}$, see \eqref{eq:flowMorphism}.
We cover $A$ by $N=N_{\delta/2}(A)$ balls of radius
$\delta/2$ and $B$ by $M=N_{\eps}(B)$ balls of radius $\eps$ and denote by
$(g_n)_{n=1}^N$ and $(h_m)_{m=1}^M$ the midpoints of these balls. 
Then the set
$\{\Gamma_{g_n,h_m}\;|\;n=1\ld N,\ m=1\ld M\}$ is $(\delta,\nu)$-spanning due to the above
and Corollary~\ref{c.no_separation}.
We obtain the estimate
\begin{eqnarray*}
  \oac_{\Fol}(\Omega(\oplam(W)),G) & = & \adjustlimits\sup_{\delta>0}\varlimsup_{\nu\to 0}\frac{\log
    \Span_{\Fol}(\Omega(\oplam(W)),G,\delta,\nu)}{-\log\nu}\\
    & \leq &\adjustlimits\sup_{\delta>0}\varlimsup_{\eps \to 0}\frac{\log(N_{\delta/2}(A)\cdot N_{\eps}(B))}{-\log
    m_H\big(\overline{B_H(\partial W,\eps)}\big)}\\
    & = &\varlimsup_{\eps \to 0}\frac{\log N_{\eps}(B)/-\log \eps}
		{\log m_H\big(\overline{B_H(\partial W,\eps})\big)/\log\eps }\\
	&\leq&\frac{\uboxdim(H)}{\uboxdim(H)-\uboxdim(\partial W)}, 
\end{eqnarray*}
where we used Minkowski's characterisation in the last step.
This completes the proof. \qed\medskip

\begin{remark}
  It is not too difficult to see that the above result is optimal in the sense
  that equality is attained for some
  examples while at the same time, it cannot hold in general.
  \begin{itemize}
  \item[(a)]
		In order to see that amorphic complexity can be smaller than
		the bound provided by \eqref{eq: upper bound ac quasicrystal}, 
		let $H=\R$ and suppose $C\ssq \R$ is an arbitrary Cantor set of
		dimension $d\in[0,1)$.
		Let $W$ be a window given by the
		union of $C$ with a countable number of 
		gaps (that is, bounded connected components of $\R\setminus C$) 
		such that $\partial W=C$.
		Clearly, this can be done such that for each $n$, we have that $W$ contains less than $n$ intervals of size $2^{-n}$ or bigger. 
		If $\eps\in(2^{-n},2^{-n+1}]$, then each of
		these intervals contributes at most $2\eps$ to $m_H(W\Delta(W+\eps))$,
		whereas the union of the other intervals contributes at most $\eps$ in total
		(and $\partial W$ does not contribute since it is of zero measure).
		Hence, we obtain $m_H(W\Delta (W+\eps))\leq 2\eps n\leq 2\eps (-\log \eps/\log 2+1)$.
		Accordingly, the computation in the proof of Theorem~\ref{t.ac_bound_for_cps}
		yields $\oac_{\Fol}(\Omega(\oplam(W)),G)\leq 1 < \frac{1}{1-d}$.
   \item[(b)]
		The most straightforward examples in which equality is attained in \eqref{eq: upper bound ac quasicrystal} are given by CPS with $H=\R$.
		We refrain from discussing the technicalities (which are in spirit similar to those in the proof of the above theorem) and simply sketch the main ingredients of the construction.
        For $\gamma>2$, consider a middle segment Cantor set $C_\gamma$ which 
		is constructed by always removing the middle $(1-2/\gamma)$-th
		part of intervals in the canonical construction of Cantor sets.
		Observe that $C_\gamma$ is of dimension 
		$\uboxdim(C_\gamma)=\log 2/\log\gamma$ with gaps
		of size $(1-2/\gamma)\cdot \gamma^{-n}$.
		If $W$ is the window that is obtained by including all gaps of size
		$(1-2/\gamma)\cdot \gamma^{-n}$ with $n$ odd, it can be readily checked that
		\[
			\lim_{\eps\to 0} \frac{\log m_H(W\Delta (W+\eps))}{\log \eps}=(1-\log 2/\log\gamma).
		\]
		We may assume without loss of generality to be given an element $(u,v)$ of 
		some set of generators of $\cL$ with $C_\gamma\ssq[0,v]$.
		Let $h_1,\ldots,h_{\lfloor 1/\eps\rfloor}\in H$ be equidistributed in $[0,v]\ssq H$.
		Similarly to the estimates in the proof of Theorem~\ref{t.ac_bound_for_cps}, it can be
		checked that for small enough $\delta$, we have that $\{\Gamma_{0,h_1},\ldots
		\Gamma_{0,h_{\lfloor 1/\eps\rfloor}}\}$ is $(\delta,\nu)$-separated with
		$\nu =m_G(B_G(0,1/\delta))m_H(W\Delta (W+\eps))$ as $\eps$ (and hence $\nu$)
		tends to zero.
		Then one obtains $\oac_{\Fol}(\Omega(\oplam(W)),G)=\sup_{\delta>0}\varlimsup_{\nu\to 0}\frac{\log
		\Sep_{\Fol}(\Omega(\oplam(W)),G,\delta,\nu)}{-\log\nu}\geq
		\frac{1}{1-\uboxdim(C_\gamma)}$.
	\item[(c)]
		Note that the construction sketched in (b) yields uncountably many
		regular model sets that lie in different conjugacy classes. 
		In fact, it shows that any value in $[1,\infty)$ can be realised as the amorphic
		complexity of a regular model set.
   \item[(d)]
		The above considerations indicate that while the structure of the
		boundary of the window inflicts some upper bound on the complexity of the
		dynamics of the resulting model set, it greatly depends on the interior of
		the window whether this bound is actually attained or not. This coincides
		with similar observations concerning the topological entropy of irregular
		model sets \cite{JaegerLenzOertel2016}. 
    \item[(e)]
    In \cite{BaakeJaegerLenz2016}, Toeplitz flows are studied from the viewpoint of aperiodic order. 
    It is shown that the canonical integer set of a Toeplitz sequence can be obtained as a model set of some canonically associated cut and project scheme. 
    Moreover, in the regular case, the box dimension of the window can be computed from the scaling behaviour of the densities of the $p$-skeletons in the construction of the Toeplitz sequence. Combined with Theorem~\ref{t.ac_bound_for_cps}, this allows us to estimate the amorphic complexity of zero entropy Toeplitz flows in terms of the densities of the $p$-skeletons.
    We refer to \cite[Theorem~2]{BaakeJaegerLenz2016} for the technical details. 
    \end{itemize}
\end{remark}

\bibliographystyle{alpha-abbrvsort}
\bibliography{lit}

\end{document}